\documentclass[11pt]{amsart}

\usepackage{amsmath}
\usepackage{amsthm}
\usepackage{hyperref}
\usepackage[margin=1.1in]{geometry}

\numberwithin{equation}{section}

\newtheorem{prop}{Proposition}
\newtheorem{lemma}[prop]{Lemma}

\newtheorem{thm}[prop]{Theorem}
\newtheorem{cor}[prop]{Corollary}

\numberwithin{prop}{section}

\theoremstyle{definition}
\newtheorem{defn}[prop]{Definition}

\newtheorem{rmk}[prop]{Remark}

\newcommand{\nc}{\newcommand}

\newcommand{\del}{\partial}
\newcommand{\delb}{\bar{\partial}}\newcommand{\dt}{\frac{\partial}{\partial t}}
\newcommand{\brs}[1]{\left| #1 \right|}
\newcommand{\gG}{\Gamma}

\newcommand{\gD}{\Delta}
\newcommand{\gd}{\delta}
\newcommand{\gs}{\sigma}

\newcommand{\gl}{\lambda}

\newcommand{\gw}{\omega}
\newcommand{\ga}{\alpha}
\newcommand{\gb}{\beta}

\renewcommand{\ge}{\epsilon}
\newcommand{\N}{\nabla}

\newcommand{\LL}{\mathcal L}

\renewcommand{\bar}[1]{\overline{#1}}

\renewcommand{\i}{\sqrt{-1}}

\newcommand{\bj}{\bar{j}}

\newcommand{\bl}{\bar{l}}

\newcommand{\bq}{\bar{q}}

\newcommand{\IP}[1]{\left<#1\right>}

\newcommand{\bgb}{\bar{\gb}}
\newcommand{\bga}{\bar{\ga}}

\newcommand{\HH}{\mathcal{H}}

\DeclareMathOperator{\Rc}{Rc}

\DeclareMathOperator{\tr}{tr}

\DeclareMathOperator{\Vol}{Vol}

\DeclareMathOperator{\Pf}{Pf}

\nc{\R}{\mathbf{R}} 
\nc{\I}{\mathbf{I}} 
\nc{\E}{\mathbf{E}} 
\nc{\B}{\mathbf{B}}
\nc{\V}{\mathbf{V}}
\nc{\J}{\mathbf{J}}
\nc{\Q}{\mathbf{Q}}
\nc{\A}{\mathbf{A}}
\nc{\W}{\mathbf{W}}
\nc{\Id}{\mathbf{1}}

\begin{document}

\title[Nondegenerate Generalized K\"ahler-Ricci flow on surfaces]{Generalized
K\"ahler-Ricci flow and the classification of
nondegenerate generalized K\"ahler surfaces}

\begin{abstract} We study the generalized K\"ahler-Ricci flow on complex
surfaces with nondegenerate Poisson structure, proving long time existence and
convergence of the flow
to a weak hyperK\"ahler structure.
\end{abstract}

\author{Jeffrey Streets}
\email{\href{mailto:jstreets@uci.edu}{jstreets@uci.edu}}
\address{Rowland Hall\\
         University of California\\
         Irvine, CA 92617}

\thanks{The author was supported by  NSF via DMS-1341836, DMS-1454854 and by the
Alfred P. Sloan Foundation via a Sloan Research Fellowship.}

\date{January 5th, 2016}

\maketitle

\section{Introduction}

Generalized K\"ahler geometry and generalized Calabi-Yau structures first arose
from investigations into
supersymmetric sigma models \cite{Gates}.  These structures were rediscovered in
the work of
Hitchin \cite{Hitchin}, growing out of a search for special geometries defined
by
volume functionals on differential forms.  The relationship between these points
of view was elaborated upon in the thesis of Gualtieri \cite{Gualtieri}.  These
structures have recently
attracted enormous interest in both the physics and mathematical communities as
natural generalizations of K\"ahler Calabi-Yau structures, inheriting a rich
physical and geometric theory.  We will focus entirely on the ``classical''
description of generalized K\"ahler geometry (cf. \cite{Gates}), i.e. not
relying on the more intrinsic point of view developed by Gualtieri
\cite{Gualtieri} using Courant algebroids.  For our purposes a generalized
K\"ahler manifold is a smooth manifold $M$ with a triple $(g, I, J)$ consisting
of two complex structures $I$ and $J$ together with a metric $g$ which is
Hermitian with respect to both.  Moreover, the two K\"ahler forms $\gw_I$ and
$\gw_J$ satisfy
\begin{align*}
 d^c_I \gw_I = H = - d^c_J \gw_J, \qquad dH = 0,
\end{align*}
where the first equation defines $H$, and $d^c_I = \i (\delb - \del)$ with
respect to the complex structure defined by $I$, and similarly for $J$.

A natural notion of Ricci flow adapted to the context of generalized K\"ahler
geometry
was introduced in work of the author and Tian \cite{STGK}.  We will call this
flow
\emph{generalized K\"ahler-Ricci flow} (GKRF).  This evolution equation was
discovered in the course of our investigations into the more general
``pluriclosed flow," \cite{ST2}, and has the interesting feature that the
complex
structures must also evolve to preserve the generalized K\"ahler condition. 
Explicitly it takes the form
\begin{gather} \label{GKRF}
 \begin{split}
 \dt g = - 2 \Rc^g + \frac{1}{2} \HH, \qquad \dt H = \gD_d H,\\
\dt I = L_{\theta_I^{\sharp}} I, \qquad \dt J = L_{\theta_J^{\sharp}} J,
\end{split}
\end{gather}
where $\HH_{ij} = H_{ipq} H_j^{pq}$, and $\theta_I, \theta_J$ are the Lee forms
of the corresponding Hermitian structures. See \S \ref{GKRFsec} for a derivation
of
these equations.  The metric and three-form component of the flow initially
arose as the renormalization group flow for nonlinear sigma models
coupled to a skew-symmetric $B$-field (cf. \cite{Polchinski}).  Supersymmetry
considerations eventually related this sigma model to generalized K\"ahler
geometry.  Given this, one might expect the renormalization group flow to
preserve generalized K\"ahler geometry.  The surprising observation of
\cite{STGK} is that 
this is indeed so, but only after introducing a further evolution equation for
the complex structures themselves.  It remains an interesting problem to derive
these evolution equations from a Langrangian-theoretic standpoint.

A central feature of generalized K\"ahler geometry, observed first by
Pontecorvo \cite{Pontecorvo}, Hitchin \cite{HitchinPoisson}, is that the tensor $g^{-1} [I,J]$
defines a
holomorphic Poisson
structure.  Previously the author studied GKRF in one of the natural ``extremes"
of
generalized K\"ahler geometry, namely when this Poisson structure vanishes.  In
this setting it was
observed that the complex structures actually remain fixed, and that the flow
reduces to a nonconvex fully nonlinear parabolic equation for a scalar potential
function \cite{SPCFSTB}.  In this paper we focus entirely on the case when this
Poisson
structure is nondegenerate, in which case we will refer to the generalized
K\"ahler structure itself as ``nondegenerate.''  It is trivial to note that GKRF
will preserve this
condition, at least for a short time, since it is an open condition.  Whereas in
the case $[I,J] = 0$ the flow of complex structures dropped out of the system,
as we will see below, the evolving complex structures essentially
determine the
entire GKRF in the nondegenerate setting.  Our main theorem gives a complete
picture of the long time existence behavior of this flow in the case of
dimension $4$, together with a rough picture of the convergence.

\begin{thm} \label{4dLTE} Let $(M^4, g, I, J)$ be a nondegenerate generalized
K\"ahler four-manifold.  The solution to generalized K\"ahler-Ricci flow with
initial data $(g,I,J)$ exists for all time.  Moreover, the associated almost
hyperK\"ahler structure $\{\gw_{K_i(t)}\}$ converges subsequentially in the
$I$-fixed gauge to a triple of closed currents $\{\gw_{K_i}^{\infty}\}$.
\end{thm}

\begin{rmk}
\begin{enumerate}
\item See Definition \ref{aktdefn} for the definition of the associated almost
hyperK\"ahler structure, and see Remark \ref{gaugermk} for the meaning of the
flow in the ``$I$-fixed gauge''.
\item The triple of limiting currents can be interpreted as a weak hyperK\"ahler
structure.  Conjecturally the flow should converge to a hyperK\"ahler metric
exponentially in the $C^{\infty}$ topology but this is not yet attainable for
technical reasons.  In the case of tori the strong convergence follows from
(\cite{SBIPCF} Theorem 1.1).
\item It had previously been observed (see Remark \ref{generalityrmk}) that one
could construct large classes of nondegenerate generalized K\"ahler structures
by special deformations of hyperK\"ahler structures.  Theorem \ref{4dLTE}
roughly indicates that this the \emph{only} way to construct such structures.
\item The solutions to GKRF in Theorem \ref{4dLTE}
are never solutions to K\"ahler-Ricci flow, unless the initial data is already
hyperK\"ahler in which case the K\"ahler-Ricci flow and generalized
K\"ahler-Ricci flow are both fixed.
\end{enumerate}
\end{rmk}

It seems natural to expect similar behavior for the generalized K\"ahler-Ricci
flow in the nondegenerate setting in all dimensions $n=4k$.  In particular, one
might expect long time existence and convergence of the flow to a generalized
Calabi-Yau structure.  In dimensions greater than $4$ it does not follow
directly that such a structure is hyperK\"ahler, and it would seem that more
general examples should exist, although we do not know of any.  While many
aspects of our proof will certainly extend to higher dimensions, some key
estimates exploit the low-dimensionality.  One important breakthrough would be
to achieve, if possible, a reduction of the flow to that of a potential
function. Local constructions \cite{potential} indicate that one can express
generalized K\"ahler structures in terms of a single potential function, but in
the nondegenerate setting the objects are described as fully nonlinear
expressions in the Hessian of the potential.  Thus it remains far from clear if
it is possible to reduce the GKRF to a scalar potential flow, as has been
achieved in the setting of vanishing Poisson structure (cf. \cite{SPCFSTB}). 
Our calculations below give hope for the possibility of such a scalar reduction,
as we show for instance that all curvature quantities involved in the flow
equations can be expressed in terms of the angle function between the complex
structures.

The proof involves a number of a priori estimates derived using the maximum
principle.  Much of the structure between the two complex structures in a
bihermitian triple $(g,I,J)$ is captured by the so-called angle function $p =
\tr(IJ)$.  As we will see in \S \ref{mainproof}, a certain function $\mu$ of the
angle satisfies the time-dependent heat equation precisely along the flow.  This
yields a priori control over the angle, and moreover a strong decay estimate for
the gradient of $\mu$.  This quantity controls the torsion, yielding a priori
decay of the torsion along the flow.  Given this estimate, we switch points of
view and study the flow merely as a solution to pluriclosed flow, and use the
reduction of the flow to a parabolic flow of a $(1,0)$-form introduced in
\cite{SBIPCF,ST3}.  In the presence of this torsion decay we can establish upper
and lower bounds for the metric depending on a certain potential function
associated to the flow.  This potential function can be shown to grow linearly,
showing 
time-dependent upper and lower bounds on the metric.  We then apply the
$C^{\ga}$ estimate on the metric established in \cite{SBIPCF} to obtain full
regularity of the flow.  Using the decay of the torsion we can derive the weak
convergence statement in the sense of currents.

Here is an outline of the rest of the paper.  In \S \ref{background} we
establish background results and notation, and also review the generalized
K\"ahler-Ricci flow.  Next in \S \ref{nondegsec} we explain fundamental
properties of nondegenerate generalized K\"ahler surfaces.  Then in \S
\ref{mainproof} we develop a number of a priori estimates for the flow.  Lastly
in \S \ref{convsec} we establish Theorem \ref{4dLTE}.

\vskip 0.1in
\textbf{Acknowledgements:} The author would like to thank 
Richard Bamler, Joel Fine, Hans-Joachim Hein, Gang Tian, Alan Weinstein, and Qi
Zhang
for helpful discussions.  Also, this work
owes a significant intellectual debt to the series of works
\cite{Bogaerts,Gates,Hull,Offshell,linearizing,potential} arising from
mathematical physics.  Moreover, I have benefited from many helpful
conversations with Martin Rocek in understanding these papers.  The
author especially thanks Marco Gualtieri for many very useful conversations on
generalized K\"ahler geometry.  Lastly, we benefited quite significantly from
Vestislav Apostolov, who provided much help in understanding his papers on
bihermitian geometry and moreover suggested obtaining convergence results in the
sense of currents.

\section{Background} \label{background}
\subsection{Notation and conventions}

In this section we fix notation, conventions, and recall some fundamental
constructions we will use in the sequel.  First, given $(M^{2n}, g, J)$ a
Hermitian manifold, let
\begin{align*}
 \gw(X,Y) = g(X,JY)
\end{align*}
be the K\"ahler form.  The Lee form is defined to be
\begin{align*}
 \theta = d^* \gw \circ J.
\end{align*}
We let $\N$ denote the Levi-Civita connection of $g$.  We will make use of two
distinct Hermitian connections, in particular the Bismut connection $\N^B$ (see
\cite{Bismut}) and
the Chern connection $\N^C$.  These are defined via
\begin{gather} \label{connections}
\begin{split}
 g(\N^B_X Y, Z) =&\ g(\N_X Y, Z) + \frac{1}{2} d^c \gw(X,Y,Z),\\
 g(\N^C_X Y, Z) =&\ g(\N_X Y, Z) + \frac{1}{2} d \gw(JX,Y,Z).
\end{split}
\end{gather}
Here $d^c = \i (\delb - \del)$, hence
\begin{align*}
 d^c \gw(X,Y,Z) = - d \gw(JX,JY,JZ).
\end{align*}
We will denote the torsion tensors by $H$ and $T$ respectively, i.e.
\begin{align*}
H(X,Y,Z) =&\ g(\N^B_X Y - \N^B_Y X - [X,Y], Z)\\
T(X,Y,Z) =&\ g(\N^C_X Y - \N^C_Y X - [X,Y], Z).
\end{align*}
Both the Bismut and Chern connections induce unitary connections on $K_M^{-1}$,
with associated
Ricci-type curvatures, which we denote via
\begin{align*}
 \rho_{B,C} (X,Y) = R_{B,C}(X,Y,J e_i, e_i).
\end{align*}

We now record some formulas for Hermitian surfaces needed in the sequel.

\begin{lemma} \label{leeformsurfaces} Let $(M^4, g, J)$ be a Hermitian surface. 
Then
 \begin{align} \label{leeformula}
  H_{ijk} = - J_i^l \theta_l \gw_{jk} - J_k^l \theta_l \gw_{ij} -
J_j^l \theta_l
\gw_{ki}.
 \end{align}
\begin{proof} First of all, for a complex surface we have the formula $d
\gw = \theta \wedge \gw$, which we express in coordinates as
\begin{align*}
 (d \gw)_{ijk} = \theta_i \gw_{jk} + \theta_k \gw_{ij} + \theta_j \gw_{ki}.
\end{align*}
Now using that $H = - d \gw(J,J,J)$, we have
\begin{align*}
 H_{ijk} =&\ - d \gw_{pqr} J_i^p J_j^q J_k^r\\
 =&\ - \left( \theta_p \gw_{qr} + \theta_r \gw_{pq} + \theta_q \gw_{rp} \right)
J_i^p J_j^q J_k^r\\
 =&\ - J_i^l \theta_l \gw_{jk} - J_k^l \theta_l \gw_{ij} - J_j^l \theta_l
\gw_{ki},
\end{align*}
as required.
\end{proof}
\end{lemma}

\begin{lemma} \label{chernlaplace} (cf. \cite{Gauduchon}) Let $(M^4, g, J)$ be a
Hermitian surface. 
Then
\begin{align*}
 \gD f =&\ \gD_C f - \IP{\theta, \N f}.
\end{align*}
\begin{proof} 
First observe that we can express in coordinates that
\begin{align*}
 J_i^p d \gw_{pjk} =&\ J_i^p \left[ \theta_p \gw_{jk} + \theta_k \gw_{pj} +
\theta_j \gw_{kp}
\right] = J_i^l \theta_l \gw_{jk} + \theta_k g_{ij} - \theta_j g_{ik}.
\end{align*}
Hence we directly compute that
\begin{align*}
 \gD f =&\ g^{ij} \N_i \N_j f\\
 =&\ g^{ij} \left(\del_i \del_j f - \gG_{ij}^k \del_k f \right)\\
 =&\ g^{ij} \left( \del_i \del_j f - (\gG^C)_{ij}^k \del_k f + \left( \gG -
\gG^C \right)_{ij}^k \del_k f \right)\\
 =&\ \gD_C f + g^{ij} \left( - \frac{1}{2} J_i^p d \gw_{pjl} g^{kl} \right) \N_k
f\\
 =&\ \gD_C f - \frac{1}{2} g^{ij} \N^l f \left(J_i^p \theta_p \gw_{jl} +
\theta_l g_{ij} - \theta_j g_{il} \right)\\
=&\ \gD_C f - \frac{1}{2} g^{ij} \N^l f \left( J_i^p \theta_p g_{j q} J_l^q +
\theta_l g_{ij} - \theta_j g_{il} \right)\\
=&\ \gD_C f - \IP{\theta, \N f}.
\end{align*}
\end{proof}
\end{lemma}

\begin{lemma} \label{gradIcalc} Let $(M^4, g, J)$ be a Hermitian surface.  Then
\begin{align*}
\N_i J_j^k =&\ \frac{1}{2} \left[ - g^{qk} \theta_q^J \gw^J_{ij} + J_j^q
\theta_q^J \gd_i^k - g^{km}
J_m^q
\theta_q^J g_{ij} - \theta_j^J J_i^k \right].
\end{align*}
\begin{proof} Since $J$ is parallel with respect to the Bismut connection we
compute using (\ref{leeformula}) that
\begin{align*}
2 \N_i J_j^k =&\ H_{ij}^l J_l^k - H_{il}^k J_j^l\\
=&\ - g^{lm} \left[ J_i^q \theta^J_q \gw^J_{jm} + J_m^q \theta^J_q \gw^J_{ij} +
J_j^q \theta^J_q \gw^J_{mi} \right] J_l^k + g^{km} \left[ J_i^q \theta^J_q
\gw^J_{lm} + J_m^q \theta^J_q \gw^J_{il} + J_l^q \theta^J_q \gw^J_{mi} \right] 
J_j^l\\
=&\ g^{lm} \left[ J_i^q \theta^J_q J_j^r g_{rm} + J_m^q \theta_q^J J_i^r g_{rj}
+ J_j^q \theta_q^J J_m^r g_{ri} \right]J_l^k - g^{km} \left[ J_i^q \theta_q^J
J_l^r g_{rm} + J_m^q \theta_q^J J_i^r g_{rl} + J_l^q \theta_q^J J_m^r g_{ri}
\right] J_j^l\\
=&\ J_i^q \theta_q^J J_j^l J_l^k + g^{qk} \theta_q^J J_i^r g_{rj} + J_j^q
\theta_q^J g^{rk} g_{ri} - J_i^q \theta^J_q J_l^k J_j^l - g^{km} J_m^q
\theta_q^J g_{ij} + g^{km} \theta_j^J J_m^r g_{ri}\\
=&\ - J_i^q \theta^J_q \gd_j^k + g^{qk} \theta_q^J J_i^r g_{rj} + J_j^q
\theta_q^J
\gd_i^k + J_i^q \theta_q^J \gd_j^k - g^{km} J_m^q \theta_q^J g_{ij} - g^{km}
\theta_j^J g_{rm} J_i^r\\
=&\ - g^{qk} \theta_q^J \gw^J_{ij} + J_j^q \theta_q^J \gd_i^k - g^{km} J_m^q
\theta_q^J g_{ij} - \theta_j^J J_i^k.
\end{align*}
\end{proof}
\end{lemma}

\subsection{Generalized K\"ahler Ricci flow} \label{GKRFsec}

In this subsection we review the construction of generalized K\"ahler Ricci flow
(GKRF) from \cite{ST3}.  To begin we review the pluriclosed flow \cite{ST2}.  
Let $(M^{2n}, g, J)$ be a Hermitian manifold as above.  We say that the
metric is \emph{pluriclosed} if
\begin{align*}
\i\del\delb \gw =  d d^c \gw = 0.
\end{align*}
In \cite{ST2} the author and Tian introduced the \emph{pluriclosed flow}
equation for such a metric,
\begin{align} \label{PCF}
 \dt \gw =&\ - (\rho_B)^{1,1} = \del \del^* \gw + \delb \delb^* \gw +
\frac{\i}{2} \del\delb \log \det g
\end{align}
where $\rho^B$ is the curvature of the determinant line bundle induced by the
Bismut connection as described above, and the $(1,1)$ superscript indicates the
projection onto the space of $(1,1)$ forms.
In \cite{ST2} we showed that this flow preserves the pluriclosed condition and
agrees with
K\"ahler-Ricci flow when the initial data is K\"ahler.
As exhibited in (\cite{ST3} Proposition 6.3), the induced pairs of metrics and
Bismut torsions $(g_t,H_t)$ satisfy
\begin{gather} \label{PCFmetricevs}
 \begin{split}
 \dt g =&\ - 2 \Rc^g + \frac{1}{2} \HH - \LL_{\theta^{\sharp}} g,\\
 \dt H =&\ \gD_d H - \LL_{\theta^{\sharp}} H,
\end{split}
\end{gather}
where $\HH_{ij} = H_{ipq} H_{j}^{pq}$.  

This crucial formula shows how to construct a flow which preserves generalized
K\"ahler
geometry.  In particular consider $(M^{2n}, g, I, J)$ a generalized K\"ahler
structure.  Then as $(g, I)$ and $(g, J)$ are pluriclosed structures, we
can construct two solutions to pluriclosed flow with these initial data,
denoting
the
K\"ahler forms $\gw^I_t, \gw^J_t$.  Then let $\phi_t^I, \phi_t^J$ denote the one
parameter families of diffeomorphisms generated by $(\theta^I)^{\sharp}_t,
(\theta^J)^{\sharp}_t$ respectively.  It
follows from (\ref{PCFmetricevs}) that both $((\phi_t^I)^* g^I_t, (\phi_t^I)^*
H^I_t)$ and $((\phi_t^J)^* g^J_t, - (\phi_t^J)^* H^J_t)$ are solutions to
\begin{gather} \label{RGflow}
 \begin{split}
 \dt g =&\ - 2 \Rc^g + \frac{1}{2} \HH,\\
 \dt H =&\ \gD_d H,
\end{split}
\end{gather}
with the same initial conditions, since the original structure was generalized
K\"ahler.  It follows that
\begin{align*}
 (\phi_t^I)^* g^I_t =&\ (\phi_t^J)^* g^J_t =: g_t\\
 (\phi_t^I)^* H^I_t =&\ - (\phi_t^J)^* H^J_t =: H_t
\end{align*}
defines a one-parameter family of generalized K\"ahler structures.  To make it
more manifest, observe that by construction certainly $g_t$ is compatible with
both of the integrable complex structures $(\phi_t^I)^* I, (\phi_t^J)^* J$. 
Thus, in principle the two complex structures must evolve to preserve the
generalized K\"ahler condition.  Making this explicit we arrive at the
generalized
K\"ahler-Ricci flow system (cf. \cite{STGK}):
\begin{gather*}
 \begin{split}
\dt g = - 2 \Rc^g + \frac{1}{2} \HH, \qquad \dt H = \gD_d H,\\
\dt I = L_{\theta_I^{\sharp}} I, \qquad \dt J = L_{\theta_J^{\sharp}} J,
\end{split}
\end{gather*}
as claimed in the introduction.

\begin{rmk} \label{gaugermk} In obtaining estimates for the flow, we will
exploit two different points of view, each of which makes performing certain
calculations easier.  Certain estimates will use the system (\ref{GKRF})
directly, which we will call a solution ``in the $B$-field gauge.''  Other times
it is easier to work with pluriclosed flow directly, so we pull back the flow to
the fixed complex manifold $(M^{2n}, I)$. 
In other words by pulling back the entire system by the family of
diffeomorphisms $(\phi_t^I)^{-1}$ we return to pluriclosed flow on $(M^{2n},
I)$, which
encodes everything about the GKRF except the other complex
structure.  But the construction above makes clear that the other complex
structure is
\begin{align*}
 J_t =&\ \left[ (\phi_t^I)^{-1} \circ \phi_t^J \right]^* J.
\end{align*}
We will refer to this point of view on GKRF as occurring ``in the $I$-fixed
gauge."
\end{rmk}

\section{Nondegenerate generalized K\"ahler surfaces} \label{nondegsec}

In this section we record some basic properties of generalized K\"ahler surfaces
with nondegenerate Poisson structure.  First we derive special linear algebraic
aspects of this structure related to the angle function, (see Definition
\ref{angledef}), which plays a central role throughout what
follows.  Next we record some background on the Poisson
structures associated to a generalized K\"ahler manifold, and its relationship
to the construction of large families of nondegenerate generalized K\"ahler
structures.  Then we exhibit some general identites for the curvature and
torsion of these structures which further emphasize the central role of the
angle function, and which are essential to the analysis to follow.

\subsection{Linear algebraic structure}

In this subsection we recall well-known fundamental linear algebraic properties
of biHermitian four-manifolds.  The low dimensionality results in some key
simplifications which are central to
the analysis to follow.

\begin{defn} \label{angledef} Given $(M^{2n}, g, I, J)$ a biHermitian manifold,
let
\begin{align*}
 p = \frac{1}{2n} \tr (I \circ J)
\end{align*}
denote the \emph{angle} between $I$ and $J$.  Observe that since $I$ and $J$ are
both
compatible with $g$, by the Cauchy-Schwarz inequality we obtain
\begin{align*}
\brs{p} = \frac{1}{2n} \brs{\IP{I,J}_g} \leq \frac{1}{2n} \brs{I}_g \brs{J}_g =
1.
\end{align*}
\end{defn}

\begin{lemma} \label{anticommutelemma} Let $(M^4, g, I, J)$ be a biHermitian
manifold where $I$ and $J$
induce the same orientation.  Then
\begin{align*}
 \{I,J\} = 2 p \Id.
\end{align*}

 \begin{proof} Since $I$ and $J$ induce the same orientation, $\gw_I$ and
$\gw_J$ are both self-dual forms.  Fix some point $x \in M$, and consider a
$g$-orthonormal basis $\gw_1, \gw_2, \gw_3$ for self-dual two forms at $x$. 
Direct calculations show that the corresponding endomorphisms given by raising
an index via $g$, call them $K_i$, all anticommute and satisfy the quaternion
relations.  Moreover, since we can express
 \begin{align*}
  F_I = a \gw_1 + b \gw_2 + c \gw_3,
 \end{align*}
with $a^2 + b^2 + c^2 = 1$, it follows that $I$ (and similarly $J$) are part of
this quaternionic structure.  In particular we may write
\begin{align*}
 I = a_I K_1 + b_I K_2 + c_I K_3, \qquad J = a_J K_1 + b_J K_2 + c_J K_3.
\end{align*}
Since the $K_i$ pairwise anticommute one then directly computes that
\begin{align*}
 \{I,J\} =&\ - 2 (a_I a_J + b_I b_J + c_I c_J) \Id.
\end{align*}
That is, $\{I,J\}$ is diagonal, and this forces the final equation by the
definition of $p$.
\end{proof}
\end{lemma}

\begin{lemma} \label{bracketcalc} Let $(M^4, g, I, J)$ be a biHermitian manifold
where $I$ and $J$
induce the same orientation.  Then
\begin{align*}
[I,J]^2 = 4 \left( p^2 - 1 \right) \Id.
\end{align*}
\begin{proof} We directly compute
\begin{align*}
 [I,J]^2 =&\ (IJ - JI)(IJ - JI)\\
 =&\ IJIJ + JIJI - 2 \Id\\
 =&\ IJ (- JI + 2p \Id) + JI (-IJ + 2p \Id) - 2 \Id\\
 =&\ 2p \{I,J\} - 4 \Id\\
 =&\ 4 \left( p^2 - 1 \right) \Id.
\end{align*}
\end{proof}
\end{lemma}

Given $(M^{2n}, g, I, J)$ a generalized K\"ahler manifold, let
\begin{align*}
 \gs = g^{-1} [I,J] \in \Lambda^2 TM.
\end{align*}
It was observed by Pontecorvo \cite{Pontecorvo} that when $n=4$, this defines a
holomorphic
Poisson structure.  This was extended to higher dimensions by Hitchin
\cite{HitchinPoisson}.  In particular, $\gs$ is of type $(2,0) + (0,2)$ with
respect
to both complex structures, and is holomorphic with respect to both complex
structures.

\begin{defn} \label{nondegdef} Given $(M^{2n}, g, I, J)$ a generalized K\"ahler
manifold, we say that it is \emph{nondegenerate} if $\gs$ defines a
nondegenerate pairing on $TM$.
\end{defn}

Observe via Lemma \ref{bracketcalc} that a generalized K\"ahler structure in
nondegenerate if and only if $\brs{p} < 1$.  A nondegenerate holomorphic Poisson
structure defines, via its inverse, a
holomorphic symplectic form $\Omega$.  Complex manifolds admitting such
structures are fairly rigid, and in particular for complex surfaces can only be
tori or $K3$.  Moreover, a nondegenerate generalized K\"ahler structure
determines an almost hyperK\"ahler structure as we next define.

\begin{defn} \label{aktdefn} Let $(M^4, g, I, J)$ be a nondegenerate generalized
K\"ahler structure.  The \emph{associated almost hyperK\"ahler structure} is the
triple $(K_0,K_1,K_2)$ where
\begin{align} \label{aktriple}
K_0 =&\ q^{-1}(IJ + p \Id), \qquad K_1 = I, \qquad  K_2 = q^{-1} (J - p I),
\end{align}
for $q = \sqrt{1-p^2}$.  Each $K_i$ is an almost complex structure compatible
with the given conformal class, and we will equivalently refer to the associated
K\"ahler forms $\{\gw_{K_i}\}$ as the almost hyperK\"ahler structure.  Direct
calculations show that any pair of $\gw_{K_i}$ satisfies
\begin{align} \label{hktriple}
\gw_{K_i} \wedge \gw_{K_i} = \gw_{K_j} \wedge \gw_{K_j}, \qquad \gw_{K_i} \wedge
\gw_{K_j} = 0.
\end{align}
Later we will need an explicit formula for the K\"ahler form associated to
$K_0$.  Direct calculations show that
\begin{align} \label{K0kahlerform}
\gw_{K_0} =&\ \Omega - i \gw_{K_1}
\end{align}
\end{defn}

As it turns out, this associated K\"ahler forms encode the almost complex
structures, as made precise in the following lemma:

\begin{lemma} (\cite{Apostolov} Lemma 5, cf. \cite{Geiges,LRC}) Let $M$ be an
oriented $4$-manifold and $\Phi_1, \Phi_2$ a pair of nondegenerate real
$2$-forms on $M$ satisfying the conditions
\begin{align*}
\Phi_1 \wedge \Phi_1 = \Phi_2 \wedge \Phi_2, \qquad \Phi_1 \wedge \Phi_2 = 0.
\end{align*}
Then there is a unique almost-complex structure $J$ on $M$ such that the
$2$-form $\Omega = \Phi_1 - \i \Phi_2$ is of type $(2,0)$ with respect to $J$. 
If moreover $\Phi_1$ and $\Phi_2$ are closed, then $J$ is integrable and
$\Omega$ defines a holomorphic symplectic structure on $(M, J)$.
\end{lemma}

\subsection{Local generality} \label{generalityrmk}

Given the existence of so much rigid holomorphic
Poisson structure, one might think that nondegenerate generalized K\"ahler
manifolds are perhaps are fully rigid,
with only finite dimensional classes of examples.  This is not the case, as was
shown by (\cite{Apostolov,Gualtieridphil}).  We follow the discussion of
(\cite{Gualtieridphil} Examples 6.31, 6.32).  There it is shown that the
specification of a nondegenerate generalized K\"ahler structure in dimension
$n=4$ with
the same orientation is equivalent to specifying three closed $2$-forms $B,
\gw_1, \gw_2$ such that
\begin{align} \label{example10}
B \wedge \gw_1 = B \wedge \gw_2 = \gw_1\wedge \gw_2 = \gw_1^2 + \gw_2^2 - 4 B^2
= 0, \qquad \gw_1^2 = \gl \gw_2^2, \qquad \gl > 0.
\end{align}
In particular, given this data, the generalized K\"ahler structure is determined
by pure spinors $e^{B + \i \gw_1}, e^{-B + \i \gw_2}$ (see \cite{Gualtieridphil}
for the pure spinor description of generalized K\"ahler structures).

One can use this interpretation to produce non-hyperHermitian nondegenerate
generalized K\"ahler structures.  In particular, start with a hyperK\"ahler
triple $(M, g, I, J, K)$, and let $F_t$ be a one-parameter family of
diffeomorphisms generated by a $\gw_K$-Hamiltonian vector field.  For $t$
sufficiently small, the forms
\begin{align*}
 B = \gw_K, \qquad \gw_1 = \gw_I - F_t^* \gw_J, \qquad \gw_2 = \gw_I + F_t^*
\gw_J
\end{align*}
satisfy (\ref{example10}).  Moreover, as
shown in (\cite{Apostolov} Lemma 6), if $f$ denotes the $\gw_K$-Hamiltonian
function generating the $\gw_K$ Hamiltonian vector field, then for the angle
function associated to the generalized K\"ahler data one has
\begin{align*}
\left. \dt p \right|_{t=0} =&\ \frac{1}{2} \gD f.
\end{align*}
Hence for any nonconstant $f$ one produces structures with nonconstant angle,
which are hence not hyperK\"ahler (cf. Lemma \ref{rigiditylemma} below).  This
same deformation was used in
\cite{Apostolov} to produce non-hyperHermitian, strongly biHermitian, conformal
classes on hyperHermitian Hopf surfaces.  However, as exhibited in
\cite{Apostolov} Corollary 2, the Gauduchon metrics in these conformal classes
are never generalized K\"ahler, and so these do not play a role in the analysis
of Theorem \ref{4dLTE}.  Note that in this example the Poisson structure is
$\gw_K$, and is crucial to the construction of the deformation.  This type of
deformation, arising from the associated Poisson structure, was generalized by
Goto \cite{Goto}.

Hence there are large families of nondegenerate generalized K\"ahler structures.
 As it turns out, it is possible to give a simple characterization of when such
a structure on a surface is hyperK\"ahler.  This lemma is generally known (cf.
\cite{Pontecorvo}), and we include a simple proof based on our curvature
calculations.

\begin{lemma} \label{rigiditylemma} Let $(M^4, g, I, J)$ be a nondegenerate
generalized K\"ahler
surface.  Then $(M^4, g)$ is hyperK\"ahler if and only if $p$ is constant.
\begin{proof} It follows from direct calculations that any two complex
structures which are part of a hyperK\"ahler sphere have constant angle. 
Conversely, if $p$ is constant then it follows from Lemma \ref{4dthetacalc}
below that $\theta_I = \theta_J = 0$, which since we are on a complex surface
implies $d \gw_I = 0$, i.e. that the metric is K\"ahler.  It then follows from
Lemma \ref{chernricci} that the metric is Calabi-Yau, hence hyperK\"ahler.
\end{proof}
\end{lemma}

\subsection{Torsion and curvature identities}

Here we record a number of useful identities for the torsion and curvature of
generalized K\"ahler manifolds.  Most of these identities have been previously
observed in the literature, but we include the short derivations for
completeness and to fix conventions/notation.

\begin{lemma}(\cite{AG} Proposition 3)  \label{leeforms} Let $(M^4, g, I, J)$ be
a generalized K\"ahler
manifold such that
$I$ and $J$ induce the same orientation.  Then $\theta^I = - \theta^J$.
\begin{proof} Note that $\gw_I$ is self-dual.  Moreover, since $I$ induces the
metric orientation the action of $I$ on forms commutes with Hodge star.  Hence
\begin{align*}
\theta^I =&\ I d^* \gw_I = I \star d \star \gw_I = I \star d \gw_I = \star I d
\gw_I = \star H_I.
\end{align*}
Similarly one obtains $\theta^J = \star H_J$.  Since $H_I = - H_J$ the result
follows.
\end{proof}
\end{lemma}

\begin{rmk} Given the result of Lemma \ref{leeforms}, to simplify notation we
will adopt the convention $\theta = \theta^I$.
\end{rmk}

\begin{lemma} (cf. \cite{Apostolov} Lemma 7) \label{dpcalc} Let $(M^4, g, I, J)$
be a nondegenerate generalized
K\"ahler
four-manifold.  Then
\begin{align*}
 d p =&\ \frac{1}{4} \left( \theta^I - \theta^J \right) [I,J] = \frac{1}{2}
\theta
[I,J].
\end{align*}
\begin{proof} By Lemma \ref{leeformsurfaces} we have
\begin{align*}
H_{ijk} = - J_i^l \theta_l \gw_{jk} - J_k^l \theta_l \gw_{ij} - J_j^l \theta_l
\gw_{ki}.
\end{align*}
Thus we compute
\begin{align*}
4 \N_a p =&\ \N_a \left( I_r^s J_s^r \right)\\
=&\ \N_a I_r^s J_s^r + I_r^s \N_a J_s^r\\
=&\ \frac{1}{2} \left( (H^I)_{a r}^t I_t^s - (H^I)_{a t}^s I_r^t \right) J_s^r +
\frac{1}{2} I_r^s
\left[ (H^J)_{a s}^t J_t^r - (H^J)_{a t}^r J_s^t \right]\\
=&\ - \frac{1}{2} \left[ g^{tv} \left( I_a^w \theta^I_w \gw^I_{rv} + I_v^w
\theta^I_w
\gw^I_{ar} + I_r^w \theta^I_w \gw^I_{va} \right) \right] I_t^s J_s^r\\
&\ + \frac{1}{2} \left[ g^{sv} \left( I_a^l \theta^I_l \gw^I_{tv} + I_v^w
\theta^I_w
\gw^I_{at} + I_t^w \theta^I_w \gw^I_{va} \right) \right] I_r^t J_s^r\\
&\ - \frac{1}{2} I_r^s \left[ g^{tv} \left( J_a^l \theta^J_l \gw^J_{sv} + J_v^l
\theta^J_l
\gw^J_{as} + J_s^l \theta^J_l \gw^J_{va} \right) \right] J_t^r\\
&\ + \frac{1}{2} I_r^s \left[ g^{rv} \left( J_a^l \theta^J_l \gw^J_{tv} + J_v^l
\theta^J_l
\gw^J_{at} + J_t^l \theta^J_l \gw_{va} \right) \right] J_s^t\\
=&\ \sum_{i=1}^{12} A_i.
\end{align*}
Direct calculations show that $A_1 = A_4 = A_7 = A_{10} = 0$.  On the other hand
we have
\begin{align*}
2 A_2 =&\ - g^{tv} I_v^w \theta_w^I \gw^I_{ar} I_t^s J_s^r= g^{tv} I_v^w
\theta_w^I
I_a^p g_{pr} I_t^s J_s^r = g^{ws} \theta^I_w I_a^p g_{pr} J_s^r = - \theta_r^I
I_a^p J_p^w = - \theta^I (JI)_a,\\
2 A_3 =&\ - g^{tv} I_r^w \theta^I_w \gw^I_{va} I_t^s J_s^r = g^{tv} I_r^w
\theta^I_w
I_v^p g_{pa} I_t^s J_s^r = g^{ps} I_r^w \theta^I_w g_{pa} J_s^r = \theta^I
(IJ)_a,\\
2 A_5 =&\ g^{sv} I_v^w \theta_w^I \gw^I_{at} I_r^t J_s^r = - g^{sv} I_v^w
\theta_w^I I_a^p g_{pt} I_r^t J_s^r = - g^{sv} I_v^w \theta^I_w g_{ar} J_s^r =
g^{sv} I_v^w \theta_w^I J_a^r g_{rs} = \theta^I (IJ)_a,\\
2 A_6 =&\ g^{sv} I_t^w \theta^I_w \gw^I_{va} I_r^t J_s^r = - g^{sv} I_t^w
\theta_w^I I_v^p g_{pa} I_r^t J_s^r = g^{sv} \theta_r^I I_v^p g_{pa} J_s^r = -
\theta_r^I I_a^s J_s^r = - \theta^I (JI)_a,\\
2 A_8 =&\ - I_r^s g^{tv} J_v^l \theta^J_l \gw^J_{as} J_t^r = I_r^s g^{tv} J_v^l
\theta^J_l J_a^p g_{ps} J_t^r = I_r^s g^{lr} \theta^J_l J_a^p g_{ps} = - I_p^l
\theta_l^J J_a^p = - \theta ^J(IJ)_a,\\
2 A_9 =&\ - I_r^s g^{tv} J_s^l \theta_l^J \gw^J_{va} J_t^r = I_r^s g^{tv} J_s^l
\theta_l^J J_v^p g_{pa} J_t^r = I_r^s g^{pr} J_s^l \theta_l^J g_{pa} = \theta^J
(JI)_a,\\
2 A_{11} =&\ I_r^s g^{rv} J_v^l \theta_l^J \gw_{at}^J J_s^t = - I_r^s g^{rv}
J_v^l \theta_l^J J_a^p g_{pt} J_s^t = - I_r^s g^{rv} J_v^l \theta_l^J g_{as} =
I_a^v J_v^l \theta^J_l = \theta^J (JI)_a,\\
2 A_{12} =&\ I_r^s g^{rv} J_t^l \theta^J_l \gw_{va} J_s^t = - I_r^s g^{rv} J_t^l
\theta^J_l J_v^p g_{pa} J_s^t = I_r^l g^{rv} \theta_l^J J_v^p g_{pa} = - I_r^l
J_a^r \theta_l^J = - \theta^J (IJ)_a.
\end{align*}
The first claimed formula follows, and the second follows from Lemma
\ref{leeforms}.
\end{proof}
\end{lemma}

\begin{lemma} \label{4dthetacalc} Given $(M^4, g, I, J)$ a nondegenerate
generalized K\"ahler
structure, one has
\begin{align*}
\theta =&\ \frac{1}{2(p^2 - 1)} dp [I,J].
\end{align*}
\begin{proof} Combining Lemmas \ref{bracketcalc} and \ref{dpcalc} we have that
\begin{align*}
  dp [I,J] =&\ \frac{1}{4} (\theta^I - \theta^J) [I,J]^2 = \frac{1}{2}
\theta [I,J]^2 =
2 (p^2 - 1) \theta.
\end{align*}
\end{proof}
\end{lemma}

\begin{lemma} \label{dptheta} Let $(M^4, g, I, J)$ be a nondegenerate
generalized K\"ahler four-manifold.  Then
\begin{enumerate}
 \item $\IP{dp,\theta} = 0$.
 \item $\brs{\theta}^2 = \frac{\brs{dp}^2}{(1 - p^2)}$
\end{enumerate}

\begin{proof} We directly compute using Lemma \ref{dpcalc},
\begin{align*}
  \IP{dp,\theta} =&\ g^{ij} dp_i \theta_j\\
  =&\ g^{ij} \left[ \theta_k [I,J]_i^k \theta_j \right]\\
  =&\ g^{ij} \left[ \theta_k \theta_j \left( I_l^k J_i^l - J_l^k I_i^l \right)
\right]\\
  =&\ \theta_k I_l^k J_i^l \theta^i + \theta_k \theta_j J_l^k g^{li} I_i^j\\
  =&\ \theta_k I_l^k J_i^l \theta^i - \theta_k \theta_j g^{kl} J_l^i I_i^j\\
  =&\ \theta_k I_l^k J_i^l \theta^i - \theta_j I_i^j J_l^i \theta^l\\
  =&\ 0.
\end{align*}
Next using Lemma \ref{4dthetacalc} we have
\begin{align*}
 \brs{\theta}^2 =&\ g^{ij} \theta_i \theta_j\\
 =&\ g^{ij} \left( \frac{1}{2 (p^2 - 1)} dp_k [I,J]_i^k \right) \left(
\frac{1}{2 (p^2 - 1)} dp_l [I,J]_j^l \right)\\
=&\ - \frac{1}{4(p^2-1)^2} dp_k g^{ki} [I,J]_i^j [I,J]_j^l dp_l\\
=&\ \frac{1}{(1-p^2)} \brs{dp}^2.
\end{align*}
\end{proof}
\end{lemma}

\begin{lemma} \label{chernricci} Let $(M^{4n}, g, I, J)$ be a nondegenerate
generalized K\"ahler
structure.  Then
\begin{align*}
(\rho_C^I)^{1,1} =&\ - d I d \log \sqrt{\det
[I,J]}.
\end{align*}
In particular, when $n=1$ we have that
\begin{align*}
 \rho_C^I =&\ - d I d \log (1-p^2).
\end{align*}
\begin{proof}
First we observe that since $\mho := \Omega^{n}$ is a holomorphic volume form,
the Chern
connection on the canonical bundle associated to the volume form $\mho \wedge
\bar{\mho}$ is flat. 
Hence
\begin{align*}
\rho^I_C(\gw_I^n) =&\  \rho^I_C(\mho \wedge \bar{\mho}) - d I d \log
\frac{\gw_I^n}{\mho \wedge \bar{\mho}} = - d I d \log
\frac{\gw_I^n}{\mho \wedge \bar{\mho}}
\end{align*}
Then we note that
\begin{align*}
\gw_I^n =&\ dV_g\\
=&\ \sqrt{\det g_{ij}} dx^1 \dots \wedge dx^{4n}\\
=&\ \sqrt{ \det \left( \Omega [I,J] \right)} dx^1 \wedge \dots \wedge dx^{4n}\\
=&\ \sqrt{\det [I,J]} \Pf \Omega\\
=&\ \sqrt{\det [I,J]} \mho \wedge \bar{\mho}.
\end{align*}
In the case $n=1$, using Lemma \ref{bracketcalc} we see
that
\begin{align*}
 \sqrt{\det [I,J]} = 4(1 - p^2).
\end{align*}
Hence the second result follows.
\end{proof}
\end{lemma}

\section{Nondegenerate Generalized Kahler Ricci flow} \label{mainproof}

In this section we derive the main a priori estimates employed in the proof of
Theorem \ref{4dLTE}.  The a priori estimates roughly
break into two parts.  First we derive evolution
equations for functions associated to the angle function in the $B$-field flow
gauge.  Very surprisingly, a certain function of the angle is a solution to the
time-dependent heat equation with no reaction terms.  Direct maximum principle
arguments based on this simple evolution equation lead to a number of strong a
priori
estimates on the torsion, which play a central role in the
proof.  Second, we study the flow in the $I$-fixed gauge, utilizing a certain
reduction of the pluriclosed flow to a flow for a $(1,0)$-form and a potential
function to obtain further a priori estimates, including uniform equivalence of
the evolving volume form.

Once these estimates are in place we can obtain the long time existence and
convergence of the flow by a familiar path.  In particular, one can exploit the
potential function to obtain an a priori estimate for the trace of the metric
with respect to a background metric.  Since we have already estimated the volume
form, this yields uniform equivalence of the metric along the flow.  Once these
are in place we can
invoke the $C^{\ga}$ estimate for the metric shown in \cite{SBIPCF} to obtain
the full regularity of the flow.  Many of these estimates are not uniform as
time goes to infinity, but we can exploit the decay of the torsion tensor to
obtain the weak convergence claims.

\subsection{A priori estimates using the angle function} \label{eveqns}

\begin{lemma} \label{pevol} Let $(M^4, g_t, I_t, J_t)$ be a solution to GKRF
with
nondegenerate initial condition in the $B$-field gauge. 
Then
\begin{align*}
 \dt p =&\ \gD p + \frac{2 p \brs{dp}^2}{(1 - p^2)}.
\end{align*}
\begin{proof} Recall that for a complex structure $J$ and a vector field $X$ we
have
\begin{align*}
 (L_X J)_k^l =&\ X^q \N_q J_k^l - J_k^p \N_p X^l + \N_k X^p J_p^l.
\end{align*}
Thus we can compute
\begin{align*}
 \dt 4 p =&\ \dt (I_j^k J_k^j)\\
 =&\ (L_{\theta} I)_j^k J_k^j - I_j^k (L_{\theta} J)_k^j\\
 =&\ \left[ \theta^p \N_p I_j^k - I_j^p \N_p \theta^k + \N_j \theta^p I_p^k
\right] J_k^j - I_j^k \left[ \theta^q \N_q J_k^j - J_k^p
\N_p \theta^j + \N_k \theta^p J_p^j \right]\\
=&\ \theta^p \N_p I_j^k J_k^j - I_j^k \theta^q \N_q J_k^j + 2 \tr \left( \N
\theta^{\sharp} \cdot [J,I] \right).
\end{align*}
We observe using Lemmas \ref{gradIcalc} and \ref{dptheta} that
\begin{align*}
\theta^p \N_p I_j^k J_k^j =&\ \theta^p \left[ - g^{qk} \theta_q^I \gw^I_{pj} +
I_j^q \theta_q^I \gd_p^k - g^{km} I_m^q
\theta_q^I g_{pj} - \theta_j^I I_p^k \right] J_k^j\\
=&\ \theta^p \left[ g^{qk} \theta_q g_{rj} I_p^r J_k^j + I_j^q \theta_q J_p^j +
g^{km} I_m^q \theta_q g_{kj} J_p^j - \theta_j I_p^k J_k^j \right]\\
=&\ - \theta_r I_j^r J_k^j \theta^k + \theta_q I_j^q J_p^j \theta^p + \theta_q
I_m^q J_p^m \theta^p - \theta_j J_k^j I_p^k \theta^p\\
=&\ \IP{\theta, \theta [I,J]}\\
=&\ 2 \IP{dp, \theta}\\
=&\ 0.
\end{align*}
A very similar calculation yields
\begin{align*}
 - I_j^k \theta^q \N_q J_k^j =&\ 0.
\end{align*}
Then we note using Lemmas \ref{bracketcalc} and \ref{4dthetacalc} that
\begin{align*}
 2 \tr \left( \N
\theta^{\sharp} \cdot [J,I] \right) =&\ 2 \N_q \theta^p [J,I]_p^q\\
=&\ 2 \N_q \theta_r g^{pr} [J,I]_p^q\\
=&\ 2 \N_q \left[ \frac{1}{2 (p^2 - 1)} \N_s p [I,J]_r^s \right] g^{pr}
[J,I]_p^q\\
=&\ \N_q \left[ \frac{1}{p^2 - 1} \N_s p [I,J]_r^s g^{pr} [J,I]_p^q
\right] - \frac{1}{p^2 - 1} \N_s p [I,J]_r^s g^{pr} \N_q [J,I]_p^q\\
=&\ \N_q \left[ \frac{1}{p^2 - 1} \N_s p g^{sr} [J,I]_r^p [J,I]_p^q
\right] - 2\theta^p \N_q [J,I]_p^q\\
=&\ 4 \N_q \N_s p g^{sq} - 2\theta^p \N_q [J,I]_p^q\\
=&\ 4 \gD p + 2\theta^p \N_q [I,J]_p^q.
\end{align*}
Now we simplify the remaining term
\begin{align*}
 2 \theta^p \N_q [I,J]_p^q =&\ 2 \theta^p \N_q \left( I_r^q J_p^r - J_r^q I_p^r
\right)\\
 =&\ 2 \theta^p \left[ (\N_q I_r^q) J_p^r + I_r^q \N_q J_p^r - (\N_q J_r^q)
I_p^r
- J_r^q (\N_q I_p^r) \right)\\
=&\ \sum_{i=1}^4 A_i.
\end{align*}
Then
\begin{align*}
 A_1 =&\ \theta^p J_p^r \left( - g^{tq} \theta_t^I \gw^I_{qr} +
I_r^t \theta_t \gd_q^q - g^{qm} I_m^t
\theta_t g_{qr} - \theta_r I_q^q \right)\\
=&\ 4 I_r^t J_p^r \theta^p \theta_t + \theta^p \theta_t J_p^r g^{tq} I_q^v
g_{vr} - \theta^p \theta_t J_p^r I_m^t g^{qm} g_{qr}\\
=&\ 4 I_r^t J_p^r \theta^p \theta_t - \theta^p \theta_t J_p^r I_r^t - \theta^p
\theta_t J_p^r I_r^t\\
=&\ 2 I_r^t J_p^r \theta^p \theta_t\\
=&\ \left( IJ + JI \right)_p^t \theta^p \theta_t\\
=&\ 2 p \brs{\theta}^2.
\end{align*}
Next
\begin{align*}
 A_2 =&\ \theta^p I_r^q \left( - g^{tr} \theta_t^J \gw^J_{qp} +
J_p^t \theta_t^J \gd_q^r - g^{rm} J_m^t
\theta_t^J g_{qp} - \theta_p^J J_q^r \right)\\
=&\ (\tr IJ) \brs{\theta}^2 - \theta^p I_r^q g^{tr} \theta_t J_q^v g_{vp} +
\theta^p I_r^q g^{rm} J_m^t \theta_t g_{qp}\\
=&\ (\tr IJ) \brs{\theta}^2 - 2 \theta^r \theta_v I_r^q J_q^v\\
=&\ 2 p \brs{\theta}^2.
\end{align*}
Also
\begin{align*}
 A_3 =&\ - \theta^p I_p^r \left( - g^{tq} \theta_t^J \gw^J_{qr} +
J_r^t \theta^J_t \gd_q^q - g^{qm} J_m^t
\theta^J_t g_{qr} - \theta^J_r J_q^q \right)\\
=&\ 4 \theta^p I_p^r J_r^t \theta_j + \theta^p I_p^r g^{tq} \theta_t J_q^v
g_{vr} - \theta^p I_p^r J_r^t \theta_t\\
=&\ 4 \theta^p I_p^r J_r^t \theta_j - \theta^p I_p^r \theta_t J_q^t g^{qv}
g_{vr} - \theta^p I_p^r J_r^t \theta_t\\
=&\ 2 \theta^p \theta_t J_r^t I_p^r\\
=&\ 2 p \brs{\theta}^2.
\end{align*}
Lastly
\begin{align*}
 A_4 =&\ - \theta^p J_r^q \left( - g^{tr} \theta_t^I \gw^I_{qp} +
I_p^t \theta_t^J \gd_q^r - g^{rm} I_m^t
\theta_t^I g_{qp} - \theta_p^I I_q^r \right)\\
=&\ \tr (IJ) \brs{\theta}^2 - \theta^p J_r^q g^{tr} \theta_t I_q^v g_{vp} +
\theta^p J_r^q g^{rm} I_m^t g_{qp} \theta_t\\
=&\ \tr(IJ) \brs{\theta}^2 - \theta_v I_q^v J_r^q \theta^r - \theta_t I_m^t
J_r^m \theta^r\\
=&\ 2 p \brs{\theta}^2.
\end{align*}
It follows that $2 \theta^p \N_q [I,J]_p^q = 8 p \brs{\theta}^2$.  Hence, also
applying Lemma \ref{dptheta} we obtain
\begin{align*}
 \dt 4p =&\ \gD 4p + 8 p \brs{\theta}^2 = \gD 4p + \frac{8 p \brs{dp}^2}{1 -
p^2}.
\end{align*}
\end{proof}
\end{lemma}

\begin{lemma} \label{logpev} Let $(M^4, g_t, I_t, J_t)$ be a solution to GKRF
with nondegenerate initial condition in the $B$-field gauge. 
Then
\begin{align*}
 \dt \log \frac{1+ p}{1- p} =&\ \gD \log \frac{1+p}{1-p}.
\end{align*}

\begin{proof} We directly compute using Lemma \ref{pevol} 
\begin{align*}
 \dt \log \frac{1+p}{1-p} =&\ \frac{1-p}{1+p} \dt \frac{1+p}{1-p}\\
 =&\ \frac{1-p}{1+p} \left[ \frac{(1-p) \dt p - (1+p) (- \dt
p)}{(1-p)^2}\right]\\
 =&\ \frac{2}{(1-p^2)} \left[ \gD p + \frac{2p}{(1-p^2)} \brs{dp}^2 \right]
\end{align*}
Similarly we have
\begin{gather} \label{logpev10}
\begin{split}
 \gD \log \frac{1+p}{1-p} =&\ \N^i \left[ \frac{1-p}{1+p} \N_i \frac{1+p}{1-p}
\right]\\
 =&\ \N^i \left\{ \frac{1-p}{1+p} \left[ \frac{(1-p) \N_i p - (1+p) (- \N_i
p)}{(1-p)^2}\right] \right\}\\
 =&\ \N^i \left\{ \frac{2}{(1-p^2)} \N_i p \right\}\\
 =&\ \frac{2}{(1-p^2)} \left[ \gD p + \frac{2p}{(1-p^2)} \brs{dp}^2 \right].
\end{split}
\end{gather}
The result follows.
\end{proof}
\end{lemma}

This very simple evolution equation leads to a number of crucial a priori
estimates for the flow, and the evolution equations themselves are very useful
in constructing test functions.  As is
well-known, for a solution to the heat equation against a Ricci flow background,
the gradient function satisfies a particularly clean evolution equation, with
the evolution of the metric exactly canceling the Ricci curvature term arising
from the Bochner formula.  For a solution to the $B$-field flow, the
contribution from the positive definite tensor $\HH$ makes the corresponding
evolution equation even more
useful.

\begin{lemma} \label{gradientevlemma} Let $(M^n, g_t, H_t)$ be a solution to
(\ref{RGflow}), and let $\phi_t$ be a solution to
\begin{align*}
 \dt \phi =&\ \gD_{g_t} \phi.
\end{align*}
Then
\begin{align*}
 \dt \brs{\N \phi}^2 =&\ \gD \brs{\N \phi}^2 - 2 \brs{\N^2 \phi}^2 - \frac{1}{2}
\IP{\HH, \N \phi \otimes \N \phi}.
\end{align*}
\begin{proof} Using the given evolution equations and the Bochner formula we
have
 \begin{align*}
  \dt \brs{\N \phi}^2 =&\ \IP{ 2 \Rc - \frac{1}{2} \HH, \N \phi \otimes \N \phi}
+ 2 \IP{\N \gD \phi, \N \phi}\\
  =&\ 2 \IP{ \gD \phi, \phi} - \frac{1}{2} \IP{\HH, \N \phi \otimes \N \phi}\\
  =&\ \gD \brs{\N \phi}^2 - 2 \brs{\N^2 \phi}^2 - \frac{1}{2} \IP{\HH, \N \phi
\otimes \N \phi},
 \end{align*}
as required.
\end{proof}
\end{lemma}

\begin{lemma} \label{Hsquaredlemma} Let $(M^4, g)$ be a Riemannian manifold, and
$H \in \Lambda^3$. Then
\begin{align*}
\HH =&\ \brs{\star H}^2 g - (\star H) \otimes (\star H).
\end{align*}
 \begin{proof} We express $H = \star \ga$, and then choose coordinates where $g$
is the identity.  It follows that
  \begin{align*}
 \HH_{ij} = H_{ipq} H_{j}^{pq} =&\ \ga^r (dV_g)_{ripq} \ga^s (dV_g)_{sjpq}.
  \end{align*}
It is clear that for any unit vector $v$ orthogonal to $\ga^{\sharp}$, one has
$\HH(v,v) = \brs{\ga}^2$.  On the other hand certainly $\HH(\ga^{\sharp},
\ga^{\sharp}) = 0$, and so the result follows.
 \end{proof}
\end{lemma}

\begin{lemma} \label{gradmuprop} Let $(M^4, g_t, I_t, J_t)$ be a solution to
GKRF with nondegenerate initial condition in the $B$-field gauge.  Furthermore
let $\mu = \log \frac{1+p}{1-p}$.  Then
\begin{align*}
 \dt \brs{\N \mu}^2 =&\ \gD \brs{\N \mu}^2 - 2 \brs{\N^2 \mu}^2 -
\frac{1-p^2}{8} \brs{\N \mu}^4 = \gD \brs{\N \mu}^2 - 2 \brs{\N^2 \mu}^2 -
\frac{2}{1-p^2} \brs{\theta}^4.
\end{align*}
 \begin{proof} Combining Proposition \ref{logpev} with Lemma
\ref{gradientevlemma} yields
 \begin{align*}
  \dt \brs{\N \mu}^2 =&\ \gD \brs{\N \mu}^2 - 2 \brs{\N^2 \mu}^2 - \frac{1}{2}
\IP{\HH, \N \mu \otimes \N \mu}.
 \end{align*}
We observe that in four dimensions, $\theta = \star H$.  Moreover, $\N \mu$ is a
multiple of $\N p$, which is orthogonal to $\theta$ via Lemma \ref{dptheta}.  It
follows from Lemma \ref{Hsquaredlemma} that
\begin{align*}
 \HH(\N \mu, \N \mu) = \brs{\theta}^2 \brs{\N \mu}^2.
\end{align*}
On the other hand using the definition of $\mu$ (cf. \ref{logpev10}) and Lemma
\ref{dptheta} we have
\begin{align*}
 \brs{\N \mu}^2 =&\ \frac{4 \brs{dp}^2}{(1-p^2)^2} = \frac{4}{1-p^2}
\brs{\theta}^2,
\end{align*}
and the result follows.
 \end{proof}
\end{lemma}

Now we derive two key a priori estimates from these evolution equations via the
maximum principle.

\begin{prop} \label{pthetaestimate} Let $(M^4, g_t, I_t, J_t)$ be a solution to
GKRF with nondegenerate initial condition in the $B$-field gauge.  Then there is
a constant $\gd = \gd(I_0,J_0)$ such that
 \begin{align*}
 -1< \inf p_0 \leq p_t \leq \sup p_0 < 1, \qquad \sup_{M \times \{t\}} \brs{\N
\mu}^2 \leq \left[
\left( \sup \brs{\N \mu_0}\right)^{-2} + \gd t \right]^{-1}.
 \end{align*}
\begin{proof} The first inequalities follow by applying the maximum principle to
the evolution equation of Lemma \ref{pevol}.  For the second we first observe
that $\frac{1}{8} \inf (1-p_t^2) \geq \frac{1}{8} \inf (1-p_0^2) = \gd > 0$. 
Then we apply the
maximum principle to the result of Lemma \ref{gradmuprop} to show that $\sup
\brs{\N \mu_t}^2$ is bounded above by the solution to the ODE
\begin{align*}
\frac{dF}{dt} =&\ - \gd F^2, \qquad F(0) = \sup \brs{\N \mu_0}^2.
\end{align*}
The proposition follows.
\end{proof}
\end{prop}

\subsection{Estimates from the decomposed pluriclosed flow} \label{decflowsec}

In this section we derive further a priori estimates for the generalized
K\"ahler-Ricci flow, purely from the point of view of pluriclosed flow.  In
\cite{ST3} the author and Tian observed that the pluriclosed flow reduces
naturally to a degenerate parabolic flow of a $(1,0)$-form.  In \cite{SBIPCF} we
exhibited a further decomposition into a
scalar flow coupled to a parabolic flow for a $(1,0)$-form, which
naturally reduces to the parabolic complex Monge-Ampere equation when the
$(1,0)$-form vanishes.  We review this construction in our special setting
below.

First, as in the reduction of K\"ahler-Ricci flow to the parabolic complex
Monge-Ampere equation (cf. \cite{TZ}), one must
choose
an appropriate family of background pluriclosed metrics whose Aeppli cohomology
classes agree with those of the flowing metric.  However, in our setting
we already know that $(M^4, I)$ admits a
holomorphic volume form.  It follows that $c_1 = 0$, and so we may choose a
Hermitian background metric $h$ such that $\rho_C(h) = 0$.  Now suppose $\gw_t$
is a solution to pluriclosed flow on $(M^4, I)$.  One can directly check using
(\ref{PCF})  (cf. \cite{SBIPCF}
Lemma 3, with $\mu = 0$) that if $\ga_t$ solves
\begin{gather} \label{alphaflow}
\begin{split}
\dt \ga =&\ \delb^*_{g_t} \gw_t - \frac{\i}{2} \del \log \frac{\det g_t}{\det
h},\\
\ga_0 =&\ 0,
\end{split}
\end{gather}
then the one-parameter family of pluriclosed metrics $\gw_{\ga} = \gw_0 + \delb
\ga + \del \bga$ is the given solution to pluriclosed flow.

For technical reasons in the proof of convergence, we will actually choose a
different initial value of $\ga$,
which corresponds to a different background metric, which is K\"ahler.  First of
all we claim that $(M^4, I)$ is indeed a K\"ahler manifold, an observation
originally appearing in (\cite{AG} Proposition 2).  By the
Enriques-Kodaira classification of surfaces, the canonical bundle being trivial
implies that $(M^4, I)$ is either a torus, a K3 surface, or a (non-K\"ahler)
primary Kodaira surface (see \cite{Barth}).  However, one can rule out the
existence of any kind of biHermitian structure (let alone generalized K\"ahler
structure) on primary Kodaira surfaces by observing that it would imply the
existence of 3 distinct harmonic self-dual forms, contradicting that $b_2^+(M) =
2$ for such a surface (see \cite{Apostolov} pg. 426 for more details).  Since we
have now shown that $(M^4, I)$ is K\"ahler, (\cite{Buchdahl} Theorem 12) asserts
that given any pluriclosed
metric $\gw_0$ on $M$, we can find $\ga_0 \in \Lambda^{1,0}$ such that $\gw :=
\gw_0 - \del \bga_0 - \delb \ga_0$ is a K\"ahler metric.  We then express
$\gw_{\ga} = \gw + \delb \left[ \ga + \ga_0 \right] + \del \left[\bar{\ga} +
\bar{\ga}_0 \right]$.  We will always make
such a choice of initial condition for $\ga$ without further comment.

The natural local decomposition of a pluriclosed metric as $\gw = \del \bga +
\delb \ga$ is not canonical, as one may observe that $\ga + \del f$ describes
the same metric, where $f \in C^{\infty}(M, \mathbb R)$.  Due to this
``gauge-invariance," the equation (\ref{alphaflow}) is not parabolic, and admits
large families of equivalent solutions.  In \cite{SBIPCF} we resolved this
ambiguity by giving a different description of (\ref{alphaflow}) which is
parabolic.  In particular, as exhibited in (\cite{SBIPCF} Proposition 3.9, in
the case the background metric is fixed and K\"ahler), if one has a family of
functions $f_t$ and $(1,0)$-forms $\gb_t$ which satisfy
\begin{gather} \label{decflow}
\begin{split}
\dt \gb =&\ \gD_{g_t} \gb - T_{g_t} \circ \delb \gb,\\
\dt f =&\ \gD_{g_t} f + \tr_{g_t} g + \log \frac{\det g_t}{\det h},\\
\ga_0 =&\ \gb_0 - \i \del f_0,
\end{split}
\end{gather}
then $\ga_t := \gb_t - \i \del f_t$ is a solution to (\ref{alphaflow}).  The
term $T \circ \delb \gb$ is defined by
\begin{align*}
(T \circ \delb \gb)_i = g^{\bl k} g^{\bq p} T_{ik\bq} \N_{\bl} \ga_p.
\end{align*}

We will use this decomposition to obtain two estimates crucial to Theorem
\ref{4dLTE}.  First we record a prior result:
\begin{lemma} \label{betaevlemma} (\cite{SBIPCF} Proposition 4.4) Given a
solution to (\ref{decflow})
as above, one has
\begin{align} \label{betanormev}
\dt \brs{\gb}^2 =&\ \gD \brs{\gb}^2 - \brs{\N \gb}^2 - \brs{\bar{\N} \gb}^2 -
\IP{Q, \gb \otimes \bar{\gb}} + 2 \Re \IP{\gb, T^{\ga} \circ \delb \gb},
\end{align}
where
\begin{align*}
Q_{i \bj} =&\ g^{\bl k} g^{\bq p}T_{i k \bq} T_{\bj \bl p}.
\end{align*}
\end{lemma}

In fact the lemma above applies in any dimension, but the next corollary
is special to $n=2$.
\begin{cor} \label{betaestcor} (\cite{SBIPCF} Corollary 4.5) Given a solution to
(\ref{decflow}) as above, one has
\begin{align} \label{betanormev2}
\dt \brs{\gb}^2 \leq&\ \gD \brs{\gb}^2 - \brs{\N \gb}^2.
\end{align}
In particular, one has
\begin{align} \label{betaest}
\sup_M \brs{\gb_t}^2_{g_t} \leq \sup_M \brs{\gb_0}^2_{g_0}.
\end{align}
\begin{proof} The estimate (\ref{betanormev2}) follows directly from
(\cite{SBIPCF} Corollary 4.5) using that the background metric is K\"ahler.  The
estimate (\ref{betaest}) follows directly from the maximum principle.
\end{proof}
\end{cor}

The estimate (\ref{betaest}) holds for any
pluriclosed flow on a
K\"ahler surface.  The next two propositions require that we are studying a
pluriclosed flow associated to a generalized K\"ahler-Ricci flow with
nondegenerate initial data.  In particular we will assume the evolution
equations and a priori estimates of \S \ref{eveqns}.

\begin{prop} \label{volumeestimate} Let $(M^4, g_t, I, J_t)$ be a solution to
GKRF with nondegenerate initial data in the $I$-fixed gauge.  Then there exists
a constant $C = C(g_0,I_0,J_0)$ such that $C^{-1} \leq
\frac{\det g}{\det g_0} \leq C$.
\begin{proof} Using Lemmas
\ref{chernlaplace}, \ref{dptheta}, and \ref{logpev} it follows that $\mu$
satisfies, in the $I$-fixed gauge, $\mu$ satisfies
\begin{align*}
\dt \mu =&\ \gD \mu - L_{\theta^{\sharp}} \mu = \gD \mu = \gD_C \mu.
\end{align*}
A simple calculation then yields
\begin{align} \label{musqev}
\dt \mu^2 =&\ \gD_C \mu^2 - 2 \brs{\N \mu}^2 = \gD_C \mu^2 - \frac{8}{1-p^2}
\brs{\theta}^2 \leq \gD_C \mu^2 - \gd \brs{\theta}^2,
\end{align}
for some universal constant $\gd$.  On the other hand, as discussed above one
has $c_1(M, I) = 0$,
and hence there exists a background Hermitian metric $h$ such that $\rho_C(h) =
0$.  Since we are in the $I$-fixed gauge, the metric is evolving by pluriclosed
flow, and hence from \cite{SBIPCF} Lemma 6.1 we conclude that
\begin{align*}
\dt \log \frac{\det g}{\det h} =&\ \gD_C \log \frac{\det g}{\det h} +
\brs{\theta}^2.
\end{align*}
Applying the maximum principle directly to this yields an a priori lower bound
for the volume form of $g$.  On the other
hand, setting $\Phi = \log \frac{\det g}{\det h} + \gd^{-1} \mu^2$ we obtain
\begin{align*}
\dt \Phi \leq&\ \gD_C \Phi.
\end{align*}
The maximum principle implies a uniform upper bound for $\Phi$, which implies a
uniform upper bound for the volume form of $g$ since $\mu$ is bounded above via
Proposition \ref{pthetaestimate}.
\end{proof}
\end{prop}

\begin{prop} \label{dtfest} Let $(M^4, g_t, I, J_t)$ be a solution to
GKRF with nondegenerate initial data in the $I$-fixed gauge.  Given a solution
to (\ref{decflow}) as above, there exists a constant $C$ depending only
on the initial data so that for any existence time $t$ one has
\begin{align*}
\sup_{M \times \{t\}} \brs{\frac{\del f}{\del t}} \leq C.
\end{align*}
\begin{proof} 
We construct a test function
\begin{align*}
 \Phi = \dt f + A_1 \brs{\gb}^2 + A_2 \brs{\N \mu}^2 + A_3 \mu^2,
\end{align*}
where the choices of constants $A_i$ will be made explicit below.  
We first compute
\begin{align*}
 \dt \dt f =&\ \dt \left[ n - \tr_{g_t} \left( \delb \gb + \del \bar{\gb}
\right) + \log \frac{\det g_t}{\det h} \right]\\
 =&\ \IP{\frac{\del g}{\del t}, \delb \gb + \del \bar{\gb}} - \tr_{g_t} \left[
\dt \left( \delb \gb + \del \bar{\gb} \right) \right] + \tr_{g_t} \frac{\del
g}{\del t}\\
 =&\ \IP{\frac{\del g}{\del t}, \delb \gb + \del \bar{\gb}} + \tr_{g_t}
\del\delb f_t\\
 =&\ \gD_{g_t}^C f_t + \IP{\frac{\del g}{\del t}, \delb \gb + \del \bar{\gb}}.
\end{align*}
Combining this with Lemmas \ref{gradmuprop}, Lemma \ref{betaevlemma} and
(\ref{musqev}) yields that there is a small constant $\gd > 0$ such that
\begin{align*}
 \left(\dt - \gD_C \right) \Phi =&\ \IP{\frac{\del g}{\del t}, \delb \gb + \del
\bar{\gb}} + A_1 \left[ - \brs{\N \gb}^2 - \brs{\bar{\N} \gb}^2 - \frac{1}{2}
\brs{T}^2 \brs{\gb}^2 + 2 \Re \IP{\gb, T^{\ga} \circ \delb \gb} \right]\\
 &\ + A_2 \left[ \theta \star \N \brs{\N \mu}^2 - 2 \brs{\N^2 \mu}^2 -
\gd \brs{\theta}^4 \right] + A_3 \left[ -\gd \brs{\theta}^2
\right].
\end{align*}
Using (\ref{betaest}) we have that
\begin{align*}
 2 \Re \IP{\gb, T^{\ga} \circ \delb \gb} \leq&\ \ge \brs{\bar{\N} \gb}^2 + C
\ge^{-1} \brs{\gb}^2 \brs{T}^2\\
 \leq&\ \ge \brs{\bar{\N} \gb}^2 + C \ge^{-1} \brs{\theta}^2.
\end{align*}
Similarly, using that $\brs{\theta}^2$ is bounded we can estimate
\begin{align*}
\theta \star \N \brs{\N \mu}^2 =&\ \N^2 \mu \star \theta^{*2}\\
\leq&\ \ge \brs{\N^2 \mu}^2 + C \ge^{-1} \brs{\theta}^4\\
\leq&\ \ge \brs{\N^2 \mu}^2 + C \ge^{-1} \brs{\theta}^2.
\end{align*}
We also note that as $\frac{\del g}{\del t}$ is expressed in terms of one
derivative of the Lee form and the Chern-Ricci curvature, it follows from Lemma
\ref{chernricci} that there is a constant $C = C((1-p^2)^{-1})$ such that
\begin{align*}
 \brs{\frac{\del g}{\del t}}^2 \leq C \left( \brs{\N^2 \mu}^2 + \brs{\theta}^4
\right).
\end{align*}
Hence we can estimate
\begin{align*}
 \brs{\IP{\frac{\del g}{\del t}, \delb \gb + \del \bar{\gb}}} \leq&\ C \left(
\brs{\N^2 \mu}^2 + \brs{\theta}^4 \right) + C \brs{\bar{\N} \gb}^2.
\end{align*}
Putting these preliminary estimates together and choosing $\ge$ sufficiently
small we obtain
\begin{align*}
 \left( \dt - \gD_C \right) \Phi \leq&\ \brs{\bar{\N} \gb}^2 \left( C -
\frac{A_1}{2} \right) + \brs{\N^2 \mu}^2 \left( C - \frac{A_2}{2} \right) +
\brs{\theta}^2 \left( C + C A_1 + C A_2 - 2 \gd A_3 \right).
\end{align*}
It is now clear that if we choose $A_1$ and $A_2$ large with respect to
controlled constants, then choose $A_3$ large with respect to controlled
constants $A_1, A_2$, and $\gd$ we obtain
\begin{align*}
 \left( \dt - \gD_C \right) \Phi \leq 0.
\end{align*}
Applying the maximum principle yields an  upper bound for $\frac{\del f}{\del
t}$, and a very similar estimate can be obtained on $-\frac{\del f}{\del t}$,
finishing the result.
\end{proof}
\end{prop}

\begin{prop} \label{ftraceest} Given the setup above, there exists a constant
$C$
depending only on the initial data such that
\begin{align*}
\brs{f} \leq C(1 + t), \qquad \brs{ \tr_{g_t} (\delb \gb + \del \bgb)} \leq C.
\end{align*}
\begin{proof} The first estimate follows directly from Proposition \ref{dtfest}.
 For the second we observe using Propositions \ref{volumeestimate} and
\ref{dtfest} that
\begin{align*}
\brs{ \tr_{g_t}(\delb \gb + \del \bgb)} =&\ \brs{n - \dt f + \log \frac{\det
g_t}{\det h}} \leq C.
\end{align*}
\end{proof}
\end{prop}

\begin{prop} \label{traceest} Given the setup above, there exists a constant $C$
depending only on the initial data such that
\begin{align*}
\tr_{g_t} g_0 \leq C e^{C(f - \inf f)}.
\end{align*}
\begin{proof} Fix some constant $A$, and let
\begin{align*}
\Phi =&\ \log \tr_{g_t} g_0 - A (f - \inf f).
\end{align*}
The function $\Phi$ is smooth on $M$, and Lipschitz in $t$ due to Proposition
\ref{dtfest}.  Using (\cite{SBIPCF} Lemma 6.2 and standard estimates, and
combining with 
(\ref{decflow}) yields
\begin{align*}
\dt \Phi \leq&\ \gD \Phi + \left(C - A \right) \tr_{g_t} g_0 + A \dt \inf f + A
\log \frac{\det g_t}{\det h}.
\end{align*}
As we noted, $\inf f$ is only Lipschitz in time, and so inequality holds in the
sense of limsups of difference quotients.  Considering this at a spatial maximum
for $\Phi$, using the result of Proposition \ref{volumeestimate}, and choosing
$A$ sufficiently large yields
\begin{align*}
\dt \Phi \leq -\frac{A}{2} \tr_{g_t} g_0 + C \leq 0,
\end{align*}
where the last line follows if the maximum value for $\Phi$, and hence
$\tr_{g_t} g_0$, is sufficiently large.  This yields an a priori upper bound for
$\Phi$, after which the result follows.
\end{proof}
\end{prop}

\section{Proof of Theorem \ref{4dLTE}} \label{convsec}

In this section we complete the proof of Theorem \ref{4dLTE}.  We first
establish the long time existence, and then prove a series of lemmas leading to
the weak convergence statement.

\begin{prop}  Let $(M^4, g, I, J)$ be a nondegenerate generalized
K\"ahler four-manifold.  The solution to generalized K\"ahler-Ricci flow with
initial data $(g,I,J)$ exists for all time.
\end{prop}

 \begin{proof} Fix $(M^4, g, I, J)$ a
nondegenerate generalized K\"ahler four-manifold.  Let $(g_t, I, J_t)$ be the
solution to GKRF with this initial data in the $I$-fixed gauge.  From
Proposition
\ref{pthetaestimate}, we have a priori estimates for $(1-p^2)^{-1}$ and
$\brs{\theta}^2$ in the $B$-field gauge.  As these are estimates on scalar
quantities associated to the time-dependent data, they hold automatically in
every gauge.  Next we choose a solution $(\gb_t, f_t)$ to the decomposed flow as
in \S \ref{decflowsec}.  Proposition
\ref{volumeestimate} provides a uniform bound for 
$\brs{\log \frac{\det g_t}{\det g_0}}$.  Moreover, it follows from Proposition
\ref{ftraceest} that $f - \inf f \leq C(1+t)$, hence Proposition \ref{traceest}
yields a uniform upper bound for $\tr_{g_t} g_0$ on any finite time interval. 
Since the volume form is already controlled, this implies $C^{-1} g_0 \leq g_t
\leq C g_0$, on $[0,T)$, for a constant $C(T)$.  We can now apply (\cite{SBIPCF}
Theorems 1.7,1.8), there are higher order estimates for $g$ on any finite time
interval.  The claim of long time existence follows from standard arguments.
\end{proof}

\begin{lemma} \label{weakconvlem10} Let $(M^4, g_t, I, J_t)$ be a solution to
(\ref{GKRF}) in the
I-fixed gauge.  Given $a \in \Lambda^1(M)$, one has
\begin{align*}
\lim_{t \to \infty} \int_M \brs{dp}_g \brs{a}_g dV_g = 0.
\end{align*}
\begin{proof} To begin we estimate
\begin{align*}
\int_M \brs{a}^2_g dV_g =&\ \int_M \i a \wedge \bar{a} \wedge \gw_{\ga}\\
\leq&\ \int_M \brs{a}^2_{\gw} \brs{ \gw_{\ga}}_{\gw} \gw \wedge \gw\\
\leq&\ C \int_M (\tr_{\gw} \gw_{\ga}) \gw \wedge \gw\\
=&\ C \int_M \gw_{\ga} \wedge \gw\\
=&\ C \int_M \gw \wedge \gw\\
=&\ C.
\end{align*}
Then, using Propositions \ref{pthetaestimate} and \ref{volumeestimate} we have
that
\begin{align*}
\int_M \brs{dp}_g \brs{a}_g dV_g \leq&\ \left( \int_M \brs{dp}^2_g dV_g
\right)^{\frac{1}{2}} \left( \int_M \brs{a}^2_g dV_g \right)^{\frac{1}{2}}\\
\leq&\ C \brs{dp}_g \Vol(g_t)\\
\leq&\ C t^{-1}.
\end{align*}
The lemma follows.
\end{proof}
\end{lemma}

\begin{prop} \label{weakconvprop20} Let $(M^4, g_t, I, J_t)$ be a solution to
(\ref{GKRF}) in the
I-fixed gauge.  Then $\{\gw_{K_i}(t) \}$ converge subsequentially as $t \to
\infty$ a triple of closed currents $\{\gw_{K_i}^{\infty} \}$.
\begin{proof} First recall that, as explained in \S \ref{decflowsec}, we know
that
there is a K\"ahler metric $\gw$ such that
\begin{align*}
 \gw_I(t) = \gw + \del \ga_t + \delb \bar{\ga}_t.
\end{align*}
It follows that
\begin{align*}
 \int_M \gw_I(t) \wedge \gw = \int_M \gw \wedge \gw \leq C.
\end{align*}
It follows from (\cite{Demailley} Chapter III Proposition 1.23) that the set
$\{\gw_I(t) \}$ is weakly compact in the sense of positive currents.  Similarly
we have 
\begin{align*}
\brs{ \int_M (\gw_J)^{1,1}_I \wedge \gw} = \brs{ \int_M p \gw_I \wedge \gw} \leq
\brs{\int_M \gw_I \wedge \gw} = C.
\end{align*}
We next claim that $(\gw_J)^{2,0 + 0,2}_I$ is bounded in $L^{\infty}$.  First we
note that it follows from Propositions \ref{pthetaestimate} and
\ref{volumeestimate} that the holomorphic symplectic structure $\Omega$ is
uniformly bounded.  Since $(\gw_J)^{2,0 + 0,2}_I = q^{-2} \Omega$, Proposition
\ref{pthetaestimate} implies the $L^{\infty}$ estimate.  Using this and
(\cite{Demailley} Chapter III Proposition 1.23) as above we see that
$\{\gw_J(t)\}$ is weakly compact in the sense of currents.  It follows directly
from (\ref{aktriple}), (\ref{K0kahlerform}), and Proposition
\ref{pthetaestimate} that $\{\gw_{K_i}(t)\}$ is weakly compact for $i= 0,1,2$. 
Hence any
sequence of times admits a subsequence converging to a triple of limiting
currents $\{\gw_{K_i}^{\infty}\}$ as claimed.

Next we show that these currents are all closed.  
We fix $a \in \Lambda^{1}$ and compute,
\begin{align*}
\brs{ \int_M \gw_{K_1}^{\infty} \wedge d a } =&\ \lim_{t \to \infty} \brs{
\int_M \gw_{K_1} \wedge
d a}\\
 =&\ \lim_{t \to \infty} \brs{ \int_M d \gw_I \wedge a}\\
\leq&\ \lim_{t \to \infty} \int_M \brs{T_g} \brs{a}_g dV_g\\
 \leq&\ \lim_{t \to \infty}\int_M \brs{dp}_g \brs{a}_g dV_g\\
=&\ 0,
 \end{align*}
 where the last line follows from Lemma \ref{weakconvlem10}.
Similarly,
\begin{align*}
\brs{ \int_M \gw_{K_2}^{\infty} \wedge d a } =&\ \lim_{t \to \infty} \brs{
\int_M \gw_{K_2} \wedge
d a}\\
 =&\ \lim_{t \to \infty} \brs{ \int_M d \left( q^{-1} \gw_J - p \gw_I \right)
\wedge a}\\
\leq&\ \lim_{t \to \infty} C( (1-p^2)^{-1}) \int_M \brs{T_g} \brs{a}_g dV_g\\
 \leq&\ \lim_{t \to \infty} C( (1-p^2)^{-1}) \int_M \brs{dp}_g \brs{a}_g dV_g\\
=&\ 0,
 \end{align*}
where again the last line follows from Lemma \ref{weakconvlem10}.  Given this,
it follows directly from (\ref{K0kahlerform}) that $d \gw_{K_0}^{\infty} = 0$. 
The proposition follows.
\end{proof}
\end{prop}

\bibliographystyle{hamsplain}

\end{document}